\documentclass[12pt, reqno]{amsart}
\usepackage{amsmath, amsthm, amscd, amsfonts, amssymb, graphicx, color}
\usepackage[bookmarksnumbered, colorlinks, plainpages]{hyperref}
\input{mathrsfs.sty}

\newtheorem{theorem}{Theorem}[section]
\newtheorem{lemma}[theorem]{Lemma}
\newtheorem{proposition}[theorem]{Proposition}
\newtheorem{corollary}[theorem]{Corollary}
\theoremstyle{definition}

\newtheorem{example}[theorem]{Example}

\theoremstyle{remark}
\newtheorem{remark}[theorem]{Remark}

\numberwithin{equation}{section}

\begin{document}

\title[Operator Ky Fan type inequalities]{Operator Ky Fan type inequalities}

\author[S. Habibzadeh, J. Rooin, M.S. Moslehian]{S. Habibzadeh$^1$, J. Rooin$^1$, \lowercase{and} M. S. Moslehian$^2$} 
\dedicatory{
Dedicated to the memory of Professor Ky Fan}
\address{$^{1}$ Department of Mathematics, Institute for Advanced Studies in Basic Sciences (IASBS), Zanjan 45137-66731, Iran}
\email{s.habibzadeh@iasbs.ac.ir}
\email{rooin@iasbs.ac.ir}

\address{$^{2}$ Department of Pure Mathematics, Ferdowsi University of Mashhad, P. O. Box 1159, Mashhad 91775, Iran}
\email{moslehian@um.ac.ir}



\begin{abstract}
In this paper, we extend some significant Ky Fan type inequalities in a large setting to operators on Hilbert spaces and derive their equality conditions. Among other things, we prove that if $f:[0,\infty)\rightarrow[0,\infty)$ is an operator monotone function with $f (1) = 1$, $f'(1)=\mu$, and associated mean $\sigma$, then for all operators $A$ and $B$ on a complex Hilbert space $\mathscr{H}$ such that $0<A,B\leq\frac{1}{2}I$, we have
\begin{equation*}
A'\nabla_\mu B'-A'\sigma B'\leq A\nabla_\mu B-A\sigma B,
\end{equation*}
where $I$ is the identity operator on $\mathscr{H}$, $A':=I-A$, $B':=I-B$, and $\nabla_\mu$ is the $\mu$-weighted arithmetic mean.	
\end{abstract}
\maketitle

\section{Introduction}
Let $n\geq2$, and let $\mu_1,\ldots,\mu_n\geq0$
such that $\sum_{i=1}^n\mu_i=1$.
For arbitrary real numbers $x_1,\ldots,x_n>0$,
we denote by $A_n, G_n,$ and $H_n$ the arithmetic mean,
the geometric mean, and the harmonic mean of $x_1,\cdots,x_n$, respectively; that is,
\begin{equation*}\label{numer-1-mean}
A_n=\sum_{i=1}^n\mu_ix_i,\qquad\qquad G_n=\prod_{i=1}^n{x_i}^{\mu_i},\qquad\qquad
H_n=\frac{1}{\sum_{i=1}^n\mu_i\frac{1}{x_i}}.
\end{equation*}
For $x_i\in(0,\frac{1}{2}]$, we denote by $A'_n, G'_n,$ and $H'_n$ the arithmetic, geometric, and harmonic means of $x'_1:=1-x_1,\cdots,x'_n:=1-x_n$, respectively; that is,
\begin{equation*}\label{numer-2-mean}
A'_n=\sum_{i=1}^n\mu_ix'_i,\qquad\qquad G'_n=\prod_{i=1}^n{x'_i}^{\mu_i},\qquad\qquad
H'_n=\frac{1}{\sum_{i=1}^n\mu_i\frac{1}{x'_i}}.
\end{equation*}
The most important and elegant Ky Fan type inequalities can be divided in the following three classes:\\
$\bullet~ \textbf{Additive class:}$
\begin{align}
A'_n-G'_n&\leq A_n-G_n,\qquad\qquad\mbox{(Alzer)}\label{ig}\\
A'_n-H'_n&\leq A_n-H_n.\qquad\qquad\mbox{(Alzer)}\label{igg}
\end{align} 
$\bullet~ \textbf{Reciprocal additive class:}$ 
\begin{align}
\frac{1}{G'_n}-\frac{1}{A'_n}&\leq\frac{1}{G_n}-\frac{1}{A_n},\qquad\qquad\qquad\label{ik}\\
\frac{1}{H'_n}-\frac{1}{G'_n}&\leq\frac{1}{H_n}-\frac{1}{G_n},\label{ikk}\\
\frac{1}{H'_n}-\frac{1}{A'_n}&\leq\frac{1}{H_n}-\frac{1}{A_n}.\label{ikkk}
\end{align} 
$\bullet~ \textbf{Multiplicative class:}$
\begin{align}
\frac{A'_n}{G'_n}&\leq\frac{A_n}{G_n},\quad\qquad\qquad\qquad\mbox{(Ky Fan)}\label{ih}\\
\frac{G'_n}{H'_n}&\leq\frac{G_n}{H_n},\quad\qquad\qquad\mbox{(Wang-Wang)}\label{ihh}\\
\frac{A'_n}{H'_n}&\leq\frac{A_n}{H_n}.\label{ihhh}
\end{align}
If $\mu_1,\ldots,\mu_n>0$, then equality holds in each of the above inequalities if and only if $x_1=\cdots=x_n$.
It should be noted that, for any $n\geq2$, $G'_n-H'_n$ and $G_n-H_n$ are not comparable in general; see \cite{Alzer-T}. 
Clearly, $(\ref{ihhh})$ is a consequence of $(\ref{ih})$ and $(\ref{ihh})$. In addition, one can easily see that $(\ref{ik})$, $(\ref{ikk})$, and $(\ref{ikkk})$ are deduced from $(\ref{ih})$, $(\ref{ihh})$, and $(\ref{ihhh})$, respectively. Inequality $(\ref{ih})$ is called Ky Fan inequality. Several mathematicians obtained extensions, refinements, and various related results. For more information on Ky Fan type inequalities, see \cite{ DEG, MO, NS, JRA, JR}, also \cite{MM, RAA}, and the references therein.

Let $\mathbb{B}(\mathscr{H})$ be the
$C^{*}$-algebra of all bounded linear operators acting on a complex Hilbert space $(\mathscr{H}, \langle\cdotp,\cdotp\rangle)$, and let $I$ be the identity operator on $\mathscr{H}$. When $\dim\mathscr{H}= n$ is finite, we shall naturally identify $\mathbb{B}(\mathscr{H})$ with the algebra $\mathcal{M}_n(\mathbb{C})$ of all $n$ by $n$ complex matrices. Given an operator $A\in\mathbb{B}(\mathscr{H})$, the spectrum of $A$ is denoted
by ${\rm sp}(A)$.
An operator
$A$ is called positive if $\langle
Ax,x\rangle\geq0$ for all $x\in\mathscr{H}$, and then we
write $A\geq0$. We denote by $\mathbb{B}^+(\mathscr{H})$ the convex cone of all positive operators on $\mathscr{H}$.
If $A$ is positive and invertible, we write $A>0$. For self-adjoint operators $A,B\in\mathbb{B}(\mathscr{H})$, we write $A\leq B$ if $B-A\geq0$. This operator order is called the L\"owner order. \\
The monotonicity principle for operator functions, concluded from (continuous) functional calculus, states that if $A\in\mathbb{B}(\mathscr{H})$ is self-adjoint and $f, g$ are continuous real-valued functions on ${\rm sp}(A)$, then
\begin{equation*}\label{9}
f(t)\geq g(t)~~~(t\in {\rm sp}(A))\mbox{ if and only if }f(A)\geq g(A).
\end{equation*}
Moreover, $f(A) = g(A)$ if and only if $f(t) = g(t)$ for all $t\in {\rm sp}(A)$.\\ 
Inequalities for real numbers may have several plausible extensions to self-adjoint operators on Hilbert spaces. Only some of them turn out to be valid. Clearly, the monotonicity principle for operator functions allows us to extend some significant inequalities for real numbers to self-adjoint operators. If such an extension of an inequality for real numbers is not possible, one may wonder whether an eigenvalue inequality is true.\par
The axiomatic theory of connections and means for pairs of positive operators have been developed by Kubo and Ando \cite{ka}. A binary operation $\sigma$ defined on the set of positive operators is called a connection provided that
\begin{itemize}
\item [\rm{(M1)}]~ $A\leq B$ and $C\leq D$ imply that $A\sigma C\leq B\sigma D$;
\item [\rm{(M2)}]~ $A_n\downarrow A$ and $B_n\downarrow B$ imply that $A_n\sigma B_n\downarrow A\sigma B$, where $A_n\downarrow A$ means that $(A_n)$ is a decreasing sequence converging to $A$ in the strong operator topology on $\mathbb{B}(\mathscr{H})$.
\item [\rm{(M3)}]~ $C^*(A\sigma B)C\leq(C^*AC)\sigma(C^*BC)$ for every operator $C$.
\end{itemize}
A mean is a connection $\sigma$ satisfying the normalized condition $I\sigma I=I$. The class of Kubo--Ando means cover many well-known operator means. For $0\leq\mu\leq1$, the weighted arithmetic mean $\nabla_\mu$, the geometric mean $\sharp_\mu$, and the harmonic mean $!_\mu$
of positive invertible operators $A$ and $B$ are defined, respectively, by
\begin{eqnarray*}
&& A\nabla_\mu B:=(1-\mu)A+\mu B,\\
&& A\sharp_\mu B:= A^{\frac{1}{2}}\big(A^{-\frac{1}{2}}
B A^{-\frac{1}{2}}\big)^{\mu}A^{\frac{1}{2}},\\
&& A!_\mu B:=\Big((1-\mu)A^{-1}+\mu B^{-1}\Big)^{-1}.
\end{eqnarray*}
In the case when $\mu=\frac{1}{2}$, the usual arithmetic,
geometric, and harmonic means of $A$ and $B$ are simply denoted by
$A\nabla B$, $A\sharp B$, and $A!B$, respectively. 
A classical arithmetic-geometric-harmonic mean inequality with respect to the  L\"owner order is
\begin{equation*}
A!_\mu B\leq A\sharp_\mu B\leq A\nabla_\mu B; 
\end{equation*}
see \cite{MSM1} for some recent developments. \\
The aim of this paper is to generalize the classical Ky Fan type inequalities $(\ref{ig})\textendash(\ref{ihhh})$ for operators on Hilbert spaces and obtain their equality conditions in the case of $n=2$.\\
In section 2, given an arbitrary operator mean $\sigma$ with representing function $f$, we establish a very useful identity for the difference of $A\nabla_\mu B-A\sigma B$, where $\mu=f'(1)$.\\
In section 3, using this identity, we extend in a large setting the Ky Fan type inequalities (\ref{ig}) and (\ref{igg}) to operators on Hilbert spaces. Imposing the separability condition on $\mathscr{H}$ and compactness of $A-B$, we obtain extended eigenvalue versions of the other above listed Ky Fan type inequalities. Also, we drive equality conditions for all presented results. 

\section{Mean identities}
In this section, we present a significant mean identity which will be used in extending the listed Ky Fan type inequalities to operators on Hilbert spaces and discussing their equality conditions.\par
By an operator monotone function, we mean a continuous real-valued function $f$ defined on an interval $J \subseteq \mathbb{R}$ such that $A\geq B$ implies that $f (A)\geq f (B)$ for all self-adjoint operators $A$ and $B$ with spectra in $J$. Some structure theorems on operator monotone functions can be found in \cite{Bh, FMPS, ka}.\\
It is an important fact that there are one-to-one correspondences between the following objects which are explained in what follows. \\
$(1)$ operator connections on the set of all positive operators on $\mathscr{H}$,\\
$(2)$ operator monotone functions from $[0,\infty)$ to $[0,\infty)$,\\
$(3)$ finite (positive) Borel measures on $[0,1]$.
\begin{theorem}[\cite{FMPS}]\label{ll}
Given a connection $\sigma$, there is a unique operator monotone function $f:[0,\infty)\rightarrow[0,\infty)$ satisfying 
\begin{equation*}
f(t)I=I\sigma(tI),\qquad t\geq0.
\end{equation*}
Moreover, the map $\sigma\mapsto f$ is bijective, affine, and order-preserving.
\end{theorem}
The operator monotone function $f$ in the previous theorem is called the representing function of $\sigma$. When emphasizing is needed, we denote by $\sigma_f$ the associated operator mean with $f$.
The connection $\sigma$ associated with a nonnegative operator monotone function $f$ on $[0,\infty)$, is formulated by 
\begin{equation}\label{F}
A\sigma B = A^{\frac{1}{2}} f(A^{-\frac{1}{2}}BA^{-\frac{1}{2}})A^{\frac{1}{2}}
\end{equation}
for all positive invertible operators $A$ and $B$. \\
The operator monotone functions $1-\mu+\mu t$, $t^\mu$, and $(1-\mu+\mu t^{-1})^{-1}$ on $[0,\infty)$ with $0\leq\mu\leq1$, correspond to the weighted arithmetic, geometric, and harmonic means, $\nabla_\mu$, $\sharp_\mu$, and $!_\mu$, respectively.\\
Let $\sigma$ be an operator mean with representing function $f$. By the integral representation of $f$, see \cite{FMPS}, it is seen that if $f$ is non-linear, then $f''(t)<0$ for all $t>0$, and in particular, $f$ is strictly concave on $[0,\infty)$.
The operator means associated with the operator monotone functions $f(t^{-1})^{-1}$, $tf(t^{-1})$ 
and $tf(t)^{-1}$ 
are called the adjoint, the transpose and the dual of $\sigma$, and denoted by $\sigma^*$, $\sigma^0$ and $\sigma^\perp$, respectively; see \cite{FMPS}. Formula (\ref{F}) gives explicit forms to the adjoint, the transpose and the dual as follows:
\begin{equation*}
A\sigma^* B=(A^{-1}\sigma B^{-1})^{-1},\qquad A\sigma^0 B=B\sigma A,\qquad A\sigma^\perp B=(B^{-1}\sigma A^{-1})^{-1},
\end{equation*}
for all positive invertible operators $A$ and $B$.
It is easy to see that $\nabla_\mu^*=!_\mu$, $(\sharp_\mu)^*=\sharp_\mu$ and $\sigma^\perp=(\sigma^*)^0=(\sigma^0)^*$.\\
Since $t^{-\frac{1}{2}}f(t^{-1})^{-1}t^{-\frac{1}{2}}=(tf(t^{-1}))^{-1}(t> 0)$,
using functional calculus at the operator $T=A^{-\frac{1}{2}}BA^{-\frac{1}{2}}$ and considering (\ref{F}), we get 
\begin{equation}\label{21}
(A\sharp B)^{-1}(A\sigma^* B)(A\sharp B)^{-1}=(A\sigma^0 B)^{-1},
\end{equation}
for all positive invertible operators $A$ and $B$.
Therefore, if $\tau$ is another operator mean, then replacing $\sigma$ with $\tau^*$ and $\sigma^*$ in (\ref{21}), respectively, we get
\begin{equation*}
(A\sharp B)^{-1}(A\tau B-A\sigma B)(A\sharp B)^{-1}=(A\tau^\perp B)^{-1}-(A\sigma^\perp B)^{-1}.
\end{equation*}
In particular, if $0\leq \mu\leq 1$, replacing $\tau$ by $\nabla_{1-\mu}$ and $\sigma$ by $\sigma^0$, we get
\begin{equation}\label{20}
(A\sharp B)^{-1}(A\nabla_{1-\mu} B-A\sigma^0 B)(A\sharp B)^{-1}=(A!_\mu B)^{-1}-(A\sigma^* B)^{-1},
\end{equation}
while replacing $\tau$ by $!_\mu$ and $\sigma$ by $\sigma^*$, we get
\begin{equation}\label{211}
(A\sharp B)^{-1}(A\sigma^* B-A!_{\mu} B)(A\sharp B)^{-1}=(A\sigma^0 B)^{-1}-(A\nabla_{1-\mu} B)^{-1},
\end{equation}
for all positive invertible operators $A$ and $B$.
\begin{remark}\label{31}
Let $f:[0,\infty)\rightarrow[0,\infty)$ be an operator monotone function with $f (1) = 1$, $f'(1)=\mu$, and associated mean $\sigma$. We know that $0\leq \mu\leq1$. 
By the concavity of $f(t)$ and $tf(t)^{-1}$, we have
\begin{equation}\label{30}
(1-\mu+\mu t^{-1})^{-1}\leq f(t)\leq 1-\mu+\mu t,\qquad t>0,
\end{equation}
which yields 
\begin{equation}\label{23}
A!_\mu B\leq A\sigma B\leq A\nabla_\mu B,
\end{equation}
for all positive invertible operators $A$ and $B$. By (\ref{30}), if $\mu=0$ or $\mu=1$, then $f\equiv1$ or $f(t)=t~(t\geq 0)$, respectively. If $f$ is non-linear, then $f$ is strictly concave with $0<\mu<1$, which in particular yields that the right inequality in (\ref{30}) is strict, when $t\neq1$.
\end{remark}
\begin{theorem}[\cite{p}]\label{l}
Given a finite Borel measure $m$ on $[0,1]$, the function 
\begin{equation}\label{1}
f(t)=\int_{[0,1]}1!_\lambda t~dm(\lambda),\qquad t\geq 0,
\end{equation}
is an operator monotone function from $[0,\infty)$ to $[0,\infty)$. In fact, every operator monotone function from $[0,\infty)$ to $[0,\infty)$ arises in this form. Moreover, the map $m\mapsto f$ is bijective, affine, and order-preserving.
\end{theorem}
The measure $m$ in the previous theorem is called the associated measure of the operator monotone function $f$.
So, a connection is a mean if and only if its associated measure is a probability measure. \\
Considering (\ref{1}) and (\ref{F}), the connection associated with a finite Borel measure $m$ on $[0,1]$ is shown as 
\begin{equation*}
A\sigma B=\int_{[0,1]}A!_\lambda B~dm(\lambda)
\end{equation*}
for all positive operators $A$ and $B$.\\
The associated measure of the weighted harmonic mean $!_\mu$ is the Dirac measure $\delta_\mu$ at $\mu$. The weighted arithmetic mean $\nabla_\mu$ has the associated measure $(1-\mu)\delta_0+\mu\delta_1$.
It is known that, for each $0<\mu<1$,
\begin{equation*}
t^\mu=\frac{\sin(\mu\pi)}{\pi}\int_{0}^{\infty}\frac{t}{\lambda+t}\lambda^{\mu-1}d\lambda,
\end{equation*} 
where the integral is taken with respect to the Lebesgue measure; see e.g. \cite{Bh}. So, the associated measure of the weighted geometric mean $\sharp_\mu$ is
\begin{equation}\label{0}
dm(\lambda)=\frac{\sin(\mu\pi)}{\pi}\cdot\frac{1}{\lambda^{1-\mu}(1-\lambda)^\mu}d\lambda.
\end{equation}
The following theorem plays an essential role in obtaining some mean inequalities and investigating their equality conditions.
\begin{theorem}\label{cc}
Let $f:[0,\infty)\rightarrow[0,\infty)$ be an operator monotone function with $f (1) = 1$, and associated mean $\sigma_f$. Then there exists an operator mean $\tau$ such that for all positive invertible operators $A, B\in\mathbb{B}^+(\mathscr{H})$, 
\begin{align}
A\nabla_{f'(1)} B-A\sigma_f B=-\frac{f''(1)}{2}(A-B)(A\tau B)^{-1}(A-B).\label{th-ex-lem-f}
\end{align}
Moreover, if $f$ is non-linear, then $\tau$ is unique with the representing function
\begin{equation}\label{77}
g(t)=-\frac{f''(1)}{2}\frac{(t-1)^2}{1-f'(1)+f'(1) t-f(t)},\qquad t\geq 0.
\end{equation}
\end{theorem}
\begin{proof}
For brevity, we set $\mu=f'(1)$ and $\sigma=\sigma_f$. If $f$ is linear, then $f(t)=1-\mu+\mu t$, and so $\sigma=\nabla_\mu$. Since $f''(1)=0$, the assertion holds by taking $\tau$ an arbitrary operator mean. \\
Now let $f$ be non-linear. Since $f''(1)\neq0$, taking $A=I$ and $B=tI~(t>0)$, we see that, in the case of existence, $\tau$ is unique.
By Theorem \ref{l}, we have
\begin{equation*}
f(t)=\int_{[0,1]}\varphi_\lambda(t)~dm(\lambda),
\end{equation*}
where $m$ is the associated probability measure of $f$, and
\begin{equation*}
\varphi_\lambda(t):=1!_\lambda t=\frac{t}{(1-\lambda) t+\lambda},\qquad\mbox{ for all } t\in(0,\infty) \mbox{ and }\lambda\in[0,1].
\end{equation*}
Notice that
\begin{equation*}
f'(t)=\int_{[0,1]}\varphi'_\lambda(t)dm(\lambda).
\end{equation*} 
A direct computation shows 
\begin{equation}\label{j}
\varphi_{\lambda}(1)+\varphi'_{\lambda}(1)(t-1)-\varphi_\lambda(t)=\lambda(1-\lambda)(t-1)^2\Big(t\nabla_\lambda 1\Big)^{-1}.
\end{equation}
By integrating from (\ref{j}), we have
\begin{equation}\label{a}
1-\mu+\mu t-f(t)=(t-1)^2\int_{[0,1]}\lambda(1-\lambda)(t\nabla_\lambda 1)^{-1}dm(\lambda),\qquad\mbox{for all}\ t\in(0,\infty).
\end{equation}
Let
\begin{equation}\label{3}
A\tau B=-\frac{f''(1)}{2}\left(\int_{[0,1]}\lambda(1-\lambda)(B\nabla_\lambda A)^{-1}dm(\lambda)\right)^{-1},
\end{equation}
for all positive invertible operators $A, B\in\mathbb{B}^+(\mathscr{H})$. It is easily seen that $\tau$ is an operator mean, and considering Theorem \ref{ll} together with the identity (\ref{a}), the function $g$ in (\ref{77}) is its representing function.
Using functional calculus in $(\ref{a})$, at each positive invertible operator $T$, we have
\begin{equation*}
(1-\mu)I+\mu T-f(T)=(T-I)\left[\int_{[0,1]}\lambda(1-\lambda)(T\nabla_\lambda I)^{-1}dm(\lambda)\right](T-I).
\end{equation*} 
Now take $T=A^{-\frac{1}{2}}BA^{-\frac{1}{2}}$ in the above identity, multiply each side of it from left and right by $A^{\frac{1}{2}}$, and consider (\ref{3}) to get (\ref{th-ex-lem-f}).
\end{proof}
\begin{corollary}
Let $0\leq\mu\leq1$. Then there exists an operator mean $\tau$ such that for all positive invertible operators $A, B\in\mathbb{B}^+(\mathscr{H})$, 
\begin{align}
A\nabla_\mu B-A\sharp_\mu B=\frac{\mu(1-\mu)}{2}(A-B)(A\tau B)^{-1}(A-B).\label{8}
\end{align}
In particular, if $\mu=\frac{1}{2}$, then
\begin{equation}
A\nabla B-A\sharp B=\frac{1}{8}(A-B)\left(\frac{A\nabla B+A\sharp B}{2}\right)^{-1}(A-B),\label{99}
\end{equation}
for all positive invertible operators $A, B\in\mathbb{B}^+(\mathscr{H})$.
\end{corollary}
\begin{proof}
The identity (\ref{8}) follows by taking $f(t)=t^\mu$ in Theorem \ref{cc}. If $\mu=\frac{1}{2}$, then $g(t)=\frac{1}{2}\left(\frac{1+t}{2}+\sqrt{t}\right)$,
and so $\tau=\frac{\nabla+\sharp}{2}$.
\end{proof}
\begin{corollary}\label{2.7}
Let $0\leq\mu\leq1$. Then for all positive invertible operators $A, B\in\mathbb{B}^+(\mathscr{H})$,
\begin{align}\label{2.8}
A\nabla_\mu B-A!_\mu B=\mu(1-\mu)(A-B)(B\nabla_\mu A)^{-1}(A-B).
\end{align}
\end{corollary}
\begin{proof}
The identity (\ref{2.8}) follows by letting $f(t)=1!_\mu t$ in Theorem \ref{cc}, and taking into account that $g(t)=1\nabla_{1-\mu}t$, when $0<\mu<1$.
\end{proof}
\begin{remark}\label{ddd}
	Let $f:[0,\infty)\rightarrow[0,\infty)$ be an operator monotone function with $f (1) = 1$, $f'(1)=\mu$, and associated mean $\sigma$. Since $f(t)\leq t$ for $t\geq1$, see \cite{FMPS}, we have
	\begin{equation}\label{11}
	\lim_{t\rightarrow+\infty}\frac{f(t)}{t^2}=0.	
	\end{equation}
	Let $h(t):=tf(t^{-1})$ be the representing function of $\sigma^0$. Clearly $h(1)=1$, and since 
	\begin{align*}
	h'(t)=f(t^{-1})-t^{-1}f'(t^{-1}),\qquad\qquad\qquad
	h''(t)=t^{-3}f''(t^{-1}),
	\end{align*}
	we have $h'(1)=1-\mu$ and $h''(1)=f''(1)$. \\
	Now, let $f$ be non-linear. Since $h''(1)<0$, $h$ is also non-linear.
	Suppose that $\tau$ and $g$ are corresponded to $f$, while  $\tau_1$ and $g_1$ to $h$, as in (\ref{th-ex-lem-f}) and (\ref{77}). By a straightforward computation, we see that $g_1(t)=tg(t^{-1})$, and so $\tau_1=\tau^0$. Thus, by Theorem \ref{cc}, for all positive invertible operators $A, B\in\mathbb{B}^+(\mathscr{H})$, we have
	\begin{align}\label{2.12}
	A\nabla_{1-\mu} B-A\sigma^0 B=-\frac{f''(1)}{2}(A-B)(B\tau A)^{-1}(A-B).
	\end{align}
	Using (\ref{11}), we get $\lim_{t\rightarrow+\infty} g(t)=+\infty$, and in particular $g'(1)\neq0$. By the same reasoning, $g_1'(1)\neq0$, which yields that $g'(1)\neq1$.
	\end{remark}
\section{Ky Fan type inequalities}
In this section, we aim to obtain some operator Ky Fan type inequalities corresponding to (\ref{ig})--(\ref{ihhh}). Also, we discus the case when equality holds in each of the obtaining inequalities. As a convention, if $A, B\in\mathbb{B}^+(\mathscr{H})$ such that $0<A,B\leq\frac{1}{2}I$, we set $A':=I-A$ and $B':=I-B$. Clearly, $A'-B'=B-A$, $A\leq A'$ and $B\leq B'$.
\subsection{Mean inequalities}
The next theorem gives us an extended natural operator version of the Ky Fan type inequalities 
(\ref{ig}) and (\ref{igg}) with its equality condition.
\begin{theorem}\label{68}
	Let $f:[0,\infty)\rightarrow[0,\infty)$ be an operator monotone function with $f(1) = 1$, $f'(1)=\mu$, and associated mean $\sigma$. If $A, B\in\mathbb{B}^+(\mathscr{H})$ such that $0<A,B\leq\frac{1}{2}I$, then 
	\begin{equation}\label{p}
	A'\nabla_\mu B'-A'\sigma B'\leq A\nabla_\mu B-A\sigma B.
	\end{equation}
	Moreover, when $f$ is non-linear, equality holds if and only if $A = B$.
\end{theorem}
\begin{proof}
Suppose that $\tau$ and $g$ are as in (\ref{th-ex-lem-f}) and (\ref{77}). Since $A',B'>0$, using (\ref{th-ex-lem-f}) for $A'$ and $B'$ instead of $A$ and $B$, respectively, we get
	\begin{equation}
	A'\nabla_\mu B'-A'\sigma B'=-\frac{f''(1)}{2}(A-B)(A'\tau B')^{-1}(A-B).\label{cor-ex-i}
	\end{equation}
	Since $\big(A'\tau B'\big)^{-1}\leq\big(A\tau B\big)^{-1}$ and $-f''(1)\geq0$, comparison of (\ref{th-ex-lem-f}) and (\ref{cor-ex-i}) yields (\ref{p}).\\
	Clearly, if $A=B$, then equality holds in (\ref{p}). Now, suppose that $f$ is non-linear and equality holds in (\ref{p}). By (\ref{th-ex-lem-f}) and (\ref{cor-ex-i}), we get
	\begin{equation}\label{24}
	(A-B)(A'\tau B')^{-1}(A-B)=(A-B)(A\tau B)^{-1}(A-B).
	\end{equation}
	Let $\gamma=g'(1)$. By Remark \ref{ddd}, we have $0<\gamma<1$. Using (\ref{23}) for the operator mean $\tau$ with its representing function $g$, we get
	\begin{equation}\label{45}
	(A'\tau B')^{-1}\leq(A'!_\gamma B')^{-1}\leq 2I\leq(A\nabla_\gamma B)^{-1}\leq(A\tau B)^{-1}.
	\end{equation}
	There are two cases: \\
	$\textbf{Case 1.}$ The function $g$ is linear. So, $\tau=\nabla_\gamma$.
	According to (\ref{24}) and (\ref{45}), we have
	\begin{equation*}
	(A'-B')(A'\nabla_\gamma B')^{-1}(A'-B')=(A'-B')(A'!_\gamma B')^{-1}(A'-B').
	\end{equation*}
	This implies that 
	\begin{equation*}
	(I-T)\left[(1-\gamma)I+\gamma T\right]^{-1}(I-T)=(1-T)\left[(1-\gamma)I+\gamma T^{-1}\right](I-T),
	\end{equation*}
	where $T=A'^{-\frac{1}{2}}B'A'^{-\frac{1}{2}}$.
	Thus 
	$\gamma(1-\gamma)(t-1)^4=0$ for all $t\in {\rm sp}(T)$. Since $\gamma\neq0,1$, we have ${\rm sp}(T)=\{1\}$, which concludes that $T=I$, and so $A=B$.\\
	$\textbf{Case 2.}$ The function $g$ is non-linear. Therefore, $g$ is strictly concave, and so $g(t)<1\nabla_\gamma t~(0<t\neq1)$.
	According to (\ref{24}) and (\ref{45}), we have
	\begin{align}
	(A-B)(A\tau B)^{-1}(A-B)=(A-B)(A\nabla_{\gamma} B)^{-1}(A-B).
	\end{align}
	This implies that 
	\begin{equation*}
	(I-T)g(T)^{-1}(I-T)=(1-T)\left[(1-\gamma)I+\gamma T\right]^{-1}(I-T),
	\end{equation*}
	where $T=A^{-\frac{1}{2}}BA^{-\frac{1}{2}}$.
	Thus 
	$(1\nabla_\gamma t-g(t))(t-1)^2=0$ for all $t\in {\rm sp}(T)$, which yields ${\rm sp}(T)=\{1\}$, and so $A=B$.
	\end{proof}
\begin{corollary}
Let $0\leq\mu\leq1$. If $A, B\in\mathbb{B}^+(\mathscr{H})$ such that $0<A,B\leq\frac{1}{2}I$, then 
	\begin{align*}
	A'\nabla_\mu B'-A'\sharp_\mu B'&\leq A\nabla_\mu B-A\sharp_\mu B,\\
	A'\nabla_\mu B'-A'!_\mu B'&\leq A\nabla_\mu B-A!_\mu B.
	\end{align*}
	Moreover, when $0<\mu<1$, equality holds in each of the above inequalities if and only if $A = B$.
\end{corollary}
The following example shows that, unfortunately except than (\ref{ig}) and (\ref{igg}), the natural operator versions of the other Ky Fan type inequalities do not hold, even in the case when operators are order comparable with each other.
\begin{example}
	Let\begin{equation*}
	A =\begin{bmatrix} \frac{1}{5} & \frac{-1}{10} \\
	\\
	\frac{-1}{10} & \frac{1}{3} \end{bmatrix},\quad\quad\\
	B= \begin{bmatrix} \frac{2}{15} & \frac{-1}{10} \\ 
	\\
	\frac{-1}{10} & \frac{1}{3} \end{bmatrix}.
	\end{equation*}
	It is easily seen that $0<B\leq A\leq\frac{1}{2}I$. Now taking $\mu=\frac{1}{2}$ and using Mathematica, we find that 
	{\scriptsize\begin{align*}
		&(A\sharp B)^{-1}-(A\nabla B)^{-1}-\left((A'\sharp B')^{-1}-(A'\nabla B')^{-1}\right)=\begin{bmatrix} 0.226844 & 0.0685098 \\
		\\
		0.0685098 & 0.0204844 \end{bmatrix}\ngeq0,
		\\
		&(A\sharp B)^{-1/2}(A\nabla B) (A\sharp B)^{-1/2}-(A'\sharp B')^{-1/2}(A'\nabla B') (A'\sharp B')^{-1/2}=\begin{bmatrix} 0.0292946 & 0.00560064 \\
		\\
		0.00560064 & 0.00101542 \end{bmatrix}\ngeq0,\\
		\\
		&(A\nabla B)^{1/2}(A\sharp B)^{-1} (A\nabla B)^{1/2}-(A'\nabla B')^{1/2}(A'\sharp B')^{-1} (A'\nabla B')^{1/2}=\begin{bmatrix} 0.0293063& 0.00556985 \\
		\\
		0.00556985 & 0.00100374 \end{bmatrix}\ngeq0.\\
		\end{align*}} 
By exactly the same data, we observe that there is no natural operator version for inequalities (\ref{ikk}), (\ref{ikkk}), (\ref{ihh}) and (\ref{ihhh}).
\end{example}
In the next part we obtain the natural eigenvalue versions for the Ky Fan type inequalities (\ref{ik})--(\ref{ihhh}). To do this, we need the following results. We state the following lemma  which can be easily achieved by the mean value theorem.
\begin{lemma}\label{h3}
If $u,v>0$ and $a<0$, then
\begin{equation}\label{h0}
av^{a-1}(u-v)\leq u^{a}-v^{a}\leq au^{a-1}(u-v).
\end{equation}
Equality holds in each inequality of (\ref{h0}) if and only if $u=v$.
\end{lemma}
\begin{proposition}\label{214}
	Let $\sigma$ and $ \tau$ be two operator means with representing functions $f$ and $g$, respectively. If $A, B\in\mathbb{B}^+(\mathscr{H})$ are two positive invertible operators, then 
{\scriptsize\begin{align}\label{22}
	(A\tau B)^{-1}\Big(A\tau B-A\sigma B\Big)(A\tau B)^{-1}\leq(A\sigma B)^{-1}-(A\tau B)^{-1}
	\leq(A\sigma B)^{-1}\Big(A\tau B-A\sigma B\Big)(A\sigma B)^{-1}.
	\end{align}}
Moreover, when $f(t)\neq g(t)~(0<t\neq 1)$, equality holds in each inequality of (\ref{22}) if and only if $A=B$.	
\end{proposition}
\begin{proof}
Taking $t>0$ and substituting $u=f(t)$, $v=g(t)$, and $a=-1$ in Lemma \ref{h3}, we obtain
{\footnotesize \begin{align}\label{k}
\big(g(t)\big)^{-1}\big[g(t)-f(t)\big]\big(g(t)\big)^{-1}\leq\big(f(t)\big)^{-1}-\big(g(t)\big)^{-1}
\leq\big(f(t)\big)^{-1}\big[g(t)-f(t)\big]\big(f(t)\big)^{-1}.
\end{align}}
Utilizing functional calculus, replacing $t$ with $A^{-\frac{1}{2}}BA^{-\frac{1}{2}}$ in (\ref{k}), and multiplying each side of them from left and right by $A^{-\frac{1}{2}}$, we get (\ref{22}). Now, let $f(t)\neq g(t)~(0<t\neq 1)$ and equality holds in the left (right) inequality of (\ref{22}). Then for each $t\in {\rm sp}(A^{-\frac{1}{2}}BA^{-\frac{1}{2}})$, equality holds in the left (right) inequality of (\ref{k}), which by Lemma \ref{h3}, yields $f(t)=g(t)$. Therefore, ${\rm sp}(A^{-\frac{1}{2}}BA^{-\frac{1}{2}})=\{1\}$, and so $A=B$.
\end{proof}
\begin{corollary}
Let $f:[0,\infty)\rightarrow[0,\infty)$ be an operator monotone function with $f (1) = 1$, $f'(1)=\mu$, and associated mean $\sigma$. If $A, B\in\mathbb{B}^+(\mathscr{H})$ are two positive invertible operators, then
{\tiny\begin{align}\label{f22}
	(A\nabla_\mu B)^{-1}\Big(A\nabla_\mu B-A\sigma B\Big)(A\nabla_\mu B)^{-1}\leq(A\sigma B)^{-1}-(A\nabla_\mu B)^{-1}
	\leq(A\sigma B)^{-1}\Big(A\nabla_\mu B-A\sigma B\Big)(A\sigma B)^{-1}.
	\end{align}}
Moreover, when $f$ is non-linear, equality holds in each inequality of (\ref{f22}) if and only if $A=B$.
\end{corollary}
\begin{remark}\label{4}
Let $f,g:[0,\infty)\rightarrow[0,\infty)$ be two different operator monotone functions with $f(1)=g(1)= 1$, and associated means $\sigma$ and $\tau$, respectively. It may happen $f$ and $g$ meet together at a point $0<t_0\neq1$, even if $f'(1)=g'(1)$. In this case, if $A=t_0^{-1}I$ and $B=I$, then ${\rm sp}(A^{-\frac{1}{2}}BA^{-\frac{1}{2}})=\{t_0\}$, and so equality holds in each of inequalities in (\ref{22}), while $A\neq B$. This shows that the hypotheses $f(t)\neq g(t)~(0<t\neq 1)$ for the case of equality in Proposition \ref{214} is not superfluous.
For an example of such functions, let $0 < r < s < 1$, $f(t) = (2t)^r$ and $g(t) = (2t)^s$ on $[0, \infty)$.
Since $f,g:[0,\infty)\rightarrow[0,\infty)$ are operator monotone functions,
$\mathfrak{B}(\mathfrak{B}(f))$ and $\mathfrak{B}(\mathfrak{B}(g))$ are operator monotone functions, where $\mathfrak{B}$ defined by
\begin{equation*}
\mathfrak{B}(f)(t) = \frac{t + f(t)}{1 + f(t)},\qquad t\geq 0,
\end{equation*}
is the Barbour transform, see \cite{Osaka}. Since $\mathfrak{B}$ is injective, $\mathfrak{B}(\mathfrak{B}(f))\neq \mathfrak{B}(\mathfrak{B}(g))$. It is easy to see that $\mathfrak{B}(\mathfrak{B}(f))(1) = 1 = \mathfrak{B}(\mathfrak{B}(g))(1)$, 
$\frac{d\mathfrak{B}(\mathfrak{B}(f))}{dt}|_{t=1} =\frac{1}{2} = \frac{d\mathfrak{B}(\mathfrak{B}(g))}{dt}|_{t=1}$,
and $\mathfrak{B}(\mathfrak{B}(f))(\frac{1}{2}) = \frac{5}{7} = \mathfrak{B}(\mathfrak{B}(g))(\frac{1}{2})$.
\end{remark}
\subsection{Eigenvalue inequalities}

In this part, we aim to obtain extended natural eigenvalue versions for the Ky Fan type inequalities (\ref{ik})--(\ref{ihhh}). Also, we discuss their equality conditions. \par
Throughout this part, let $(\mathscr{H}, \langle\cdotp,\cdotp\rangle)$ denote a complex separable Hilbert space. For a
compact positive operator $A\in\mathbb{B}^+(\mathscr{H})$, let $\lambda_1(A)\geq\lambda_2(A)\geq\cdots\geq0$ denote the eigenvalues of $A$
arranged in decreasing order and repeated according to their multiplicity. To prove eigenvalue inequalities, the following lemmas will be used in what follows. 
\begin{lemma}[{\cite{HK}}]\label{p1}
If $A,B\in\mathbb{B}^+(\mathscr{H})$ are positive such that $A$ is compact and $A\geq B$, then so is $B$.
\end{lemma}
\begin{lemma}[{\cite[p. 26]{G}}]\label{p2}
If $A,B\in\mathbb{B}^+(\mathscr{H})$ are positive operators such that $A$ is compact and $A\geq B$, then $\lambda_j(A)\geq\lambda_j(B)$ for all $j=1,2,\ldots$.
\end{lemma}
This fact is known as the Weyl's monotonicity principle for compact positive operators.\\
It is easily seen that if $A\in\mathbb{B}(\mathscr{H})$ and $\lambda\in\mathbb{C}\backslash\{0\}$, then $\dim\ker(A^*A-\lambda I)=\dim\ker(AA^*-\lambda I)$.
This together with the basic fact when $\lambda=0$, conclude the following lemma.
\begin{lemma}\label{p3}
Let $A\in\mathbb{B}(\mathscr{H})$. If $A$ is compact, or one of the operators $A^*A- I$ or $AA^*- I$ is compact and positive, then $\lambda_j(A^*A)=\lambda_j(AA^*)$ for all $j=1,2,\ldots$.
\end{lemma}
\begin{lemma}[{\cite[p. 26]{G}}]\label{pi} 
Let $A,B\in\mathbb{B}^+(\mathscr{H})$ be compact positive operators such that $A\geq B$. Then $A = B$ if and only if $\lambda_j(A)=\lambda_j(B)$ for all $j=1,2,\ldots$.
\end{lemma} 
Let $A,B\in\mathbb{B}^+(\mathscr{H})$ be such that $A-B$ is compact. Let $\sigma$ be an operator mean associated with an operator monotone function $f$ on $[0,\infty)$ such that $f'(1) = \mu$. So considering (\ref{th-ex-lem-f}), we infer that the operators $A\nabla_\mu B-A\sigma B$ and $A\nabla_{1-\mu} B-A\sigma^0 B$ are also compact. 
The positivity and compactness of the operators $(A\sigma B)^{-1}-(A\nabla_\mu B)^{-1}$ and $(A!_\mu B)^{-1}-(A\sigma^* B)^{-1}$ follow from (\ref{f22}) and (\ref{20}), respectively. These imply that $(A\nabla_\mu B)^{1/2}(A\sigma B)^{-1}(A\nabla_\mu B)^{1/2}-I$ and $(A\sigma^* B)^{1/2}(A!_\mu B)^{-1}(A\sigma^* B)^{1/2}-I$ are positive and compact.
Further, we conclude positivity and compactness of operators $(A\sigma B)^{-1/2}(A\nabla_\mu B) (A\sigma B)^{-1/2}-I$ and $(A!_\mu B)^{-1/2}(A\sigma^* B) (A!_\mu B)^{-1/2}-I$, by (\ref{th-ex-lem-f}) and (\ref{211}). So, the eigenvalues of operators
$(A\nabla_\mu B)^{1/2}(A\sigma B)^{-1}(A\nabla_\mu B)^{1/2}$, $(A\sigma^* B)^{1/2}(A!_\mu B)^{-1}(A\sigma^* B)^{1/2}$, and $(A\sigma B)^{-1/2}(A\nabla_\mu B) (A\sigma B)^{-1/2}$ as well as the eigenvalues of operator $(A!_\mu B)^{-1/2}(A\sigma^* B) (A!_\mu B)^{-1/2}$ are numerable and the multiplicities of their eigenvalues, which are greater than $1$, are finite. Therefore, we can again arrange their eigenvalues in decreasing order according to their multiplicities. Similar assertions hold if we replace $A$ and $B$ by $A'$ and $B'$, respectively. 
First, we are going to prove eigenvalue inequalities for the reciprocal additive Ky Fan type inequalities (\ref{ik})--(\ref{ikkk}).
\begin{theorem}\label{69}
Let $f:[0,\infty)\rightarrow[0,\infty)$ be an operator monotone function with $f (1) = 1$, $f'(1)=\mu$, and associated mean $\sigma$. If $A, B\in\mathbb{B}^+(\mathscr{H})$ such that $0<A,B\leq\frac{1}{2}I$ and $A-B$ is compact, then for all $j=1,2,\ldots$,	
\begin{align}
&\lambda_j((A'\sigma B')^{-1}-(A'\nabla_\mu B')^{-1})\leq\lambda_j((A\sigma B)^{-1}-(A\nabla_\mu B)^{-1}),\label{h00}\\
&\lambda_j((A'!_\mu B')^{-1}-(A'\sigma^* B')^{-1})\leq\lambda_j((A!_\mu B)^{-1}-(A\sigma^* B)^{-1}).\label{h001}
\end{align}
Moreover, when $f$ is non-linear, equality holds in each of the above systems of inequalities if and only if $A = B$.
\end{theorem}
\begin{proof}
For each $j=1,2,\ldots$, we have
\begin{align*}
&\lambda_j\big((A'\sigma B')^{-1}-(A'\nabla_\mu B')^{-1}\big)\\
&\leq\lambda_j\big((A'\sigma B')^{-1}\big(A'\nabla_\mu B'-A'\sigma B'\big)(A'\sigma B')^{-1}\big)\tag{by (\ref{f22})}
\\
&=\lambda_j\big(\big(A'\nabla_\mu B'-A'\sigma B'\big)^{\frac{1}{2}}(A'\sigma B')^{-2}\big(A'\nabla_\mu B'-A'\sigma B'\big)^{\frac{1}{2}}\big)\tag{by Lemma  \ref{p3}}\\
&\leq 4\lambda_j(A'\nabla_\mu B'-A'\sigma B')\tag{since $(A'\sigma B')^{-2}\leq4I$}
\\
&\leq 4\lambda_j(A\nabla_\mu B-A\sigma B)\tag{by (\ref{p}) and Lemma \ref{p2}}
\\
&\leq\lambda_j\big(\big(A\nabla_\mu B-A\sigma B\big)^{\frac{1}{2}}(A\nabla_\mu B)^{-2}\big(A\nabla_\mu B-A\sigma B\big)^{\frac{1}{2}}\big)\tag{since $4I\leq (A\nabla_\mu B)^{-2}$}
\\
&=\lambda_j\big((A\nabla_\mu B)^{-1}\big(A\nabla_\mu B-A\sigma B\big)(A\nabla_\mu B)^{-1}\big)\tag{by Lemma \ref{p3}}\\
&\leq\lambda_j\big((A\sigma B)^{-1}-(A\nabla_\mu B)^{-1}\big),\tag{by (\ref{f22})}
\end{align*}
and so the system of inequalities (\ref{h00}) holds.
Now, let $f$ be non-linear and equality holds in (\ref{h00}). Then, by the above computation, we get
\begin{equation*}
\lambda_j(A'\nabla_\mu B'-A'\sigma B')=\lambda_j(A\nabla_\mu B-A\sigma B),\qquad j=1,2,\ldots.
\end{equation*}
It follows from (\ref{p}) and Lemma \ref{pi}, 
\begin{equation*}
A'\nabla_\mu B'-A'\sigma B'=A\nabla_\mu B-A\sigma B,
\end{equation*}
which by Theorem \ref{68}, we conclude that $A=B$.\\
The system of inequalities (\ref{h001}) and its equality condition follows from (\ref{20}) by a similar argument. 
\end{proof}
\begin{corollary}
Let $0\leq\mu\leq1$. If $A, B\in\mathbb{B}^+(\mathscr{H})$ such that $0<A,B\leq\frac{1}{2}I$ and $A-B$ is compact,  
then for all $j=1,2,\ldots$,	
\begin{align*}
&\lambda_j((A'\sharp_\mu B')^{-1}-(A'\nabla_\mu B')^{-1})\leq\lambda_j((A\sharp_\mu B)^{-1}-(A\nabla_\mu B)^{-1}),\\
&\lambda_j((A'!_\mu B')^{-1}-(A'\sharp_\mu B')^{-1})\leq\lambda_j((A!_\mu B)^{-1}-(A\sharp_\mu B)^{-1}),\\
&\lambda_j((A'!_\mu B')^{-1}-(A'\nabla_\mu B')^{-1})\leq\lambda_j((A!_\mu B)^{-1}-(A\nabla_\mu B)^{-1}).
\end{align*}
Moreover, when $0 < \mu < 1$, equality holds in each of the above systems of inequalities if and only if $A = B$.
\end{corollary}
Next, we are going to prove eigenvalue inequalities for the multiplicative Ky Fan type inequalities (\ref{ih})--(\ref{ihhh}). 
\begin{theorem}
Let $f:[0,\infty)\rightarrow[0,\infty)$ be an operator monotone function with $f (1) = 1$, $f'(1)=\mu$, and associated mean $\sigma$. If $A, B\in\mathbb{B}^+(\mathscr{H})$ such that $0<A,B\leq\frac{1}{2}I$ and $A-B$ is compact, then for all $j=1,2,\ldots$,
{\footnotesize \begin{align}
&\lambda_j\left((A'\sigma B')^{-\frac{1}{2}}(A'\nabla_\mu B')(A'\sigma B')^{-\frac{1}{2}}\right)\leq\lambda_j\left((A\sigma B)^{-\frac{1}{2}}(A\nabla_\mu B)(A\sigma B)^{-\frac{1}{2}}\right),\label{dd}\\
&\lambda_j\left((A'!_\mu B')^{-\frac{1}{2}}(A'\sigma^* B')(A'!_\mu B')^{-\frac{1}{2}}\right)\leq\lambda_j\left((A!_\mu B)^{-\frac{1}{2}}(A\sigma^* B)(A!_\mu B)^{-\frac{1}{2}}\right).\label{d3}
\end{align} }
Moreover, when $f$ is non-linear, equality holds in each of the above systems of inequalities if and only if $A = B$.
\end{theorem}
\begin{proof}
For each $j=1,2,\ldots$, we have
{\footnotesize\begin{align*}
&\lambda_j\left((A'\sigma B')^{-\frac{1}{2}}(A'\nabla_\mu B')(A'\sigma B')^{-\frac{1}{2}}\right)\\
&=\lambda_j\left((A'\sigma B')^{-\frac{1}{2}}(A'\nabla_\mu B')(A'\sigma B')^{-\frac{1}{2}}-I\right)+1
\\
&=-\frac{f''(1)}{2}\lambda_j\left((A'\sigma B')^{-\frac{1}{2}}(A-B)(A'\tau B')^{-1}(A-B)(A'\sigma B')^{-\frac{1}{2}}\right)+1\tag{by Theorem \ref{cc}}
\\
&=-\frac{f''(1)}{2}\lambda_j\left((A'\tau B')^{-\frac{1}{2}}(A-B)(A'\sigma B')^{-1}(A-B)(A'\tau B')^{-\frac{1}{2}}\right)+1\tag{by Lemma \ref{p3}}
\\
&\leq-f''(1)\lambda_j\left((A'\tau B')^{-\frac{1}{2}}(A-B)^2(A'\tau B')^{-\frac{1}{2}}\right)+1\tag{since $(A'\sigma B')^{-1}\leq 2I$}
\\
&=2\lambda_j\left(-\frac{f''(1)}{2}(A-B)(A'\tau B')^{-1}(A-B)\right)+1\tag{by Lemma \ref{p3}}\\
&=2\lambda_j\left(A'\nabla_\mu B'-A'\sigma B'\right)+1\tag{by Theorem \ref{cc}}
\\
&\leq2\lambda_j\left(A\nabla_\mu B-A\sigma B\right)+1\tag{by (\ref{p}) and Lemma \ref{p2}}
\\
&=2\lambda_j\left(-\frac{f''(1)}{2}(A-B)(A\tau B)^{-1}(A-B)\right)+1\tag{by Theorem \ref{cc}}\\
&=-f''(1)\lambda_j\left((A\tau B)^{-\frac{1}{2}}(A-B)^2(A\tau B)^{-\frac{1}{2}}\right)+1\tag{by Lemma \ref{p3}}\\
&\leq-\frac{f''(1)}{2}\lambda_j\left((A\tau B)^{-\frac{1}{2}}(A-B)(A\sigma B)^{-1}(A-B)(A\tau B)^{-\frac{1}{2}}\right)+1\tag{since $2I\leq(A\sigma B)^{-1}$}\\
&=-\frac{f''(1)}{2}\lambda_j\left((A\sigma B)^{-\frac{1}{2}}(A-B)(A\tau B)^{-1}(A-B)(A\sigma B)^{-\frac{1}{2}}\right)+1\tag{by Lemma \ref{p3}}\\
&=\lambda_j\left((A\sigma B)^{-\frac{1}{2}}(A\nabla_\mu B)(A\sigma B)^{-\frac{1}{2}}-I\right)+1\tag{by Theorem \ref{cc}}\\
&=\lambda_j\left((A\sigma B)^{-\frac{1}{2}}(A\nabla_\mu B)(A\sigma B)^{-\frac{1}{2}}\right).
\end{align*}}
Thus, the system of inequalities (\ref{dd}) holds.
Now, let $f$ be non-linear and equality holds in (\ref{dd}). Then, by the above computation, we get
\begin{equation*}
\lambda_j(A'\nabla_{\mu} B'-A'\sigma B')=\lambda_j(A\nabla_{\mu} B-A\sigma B),\qquad j=1,2,\ldots,
\end{equation*}
which by the same reasoning as in the proof of Theorem \ref{69}, we ensure that $A = B$.\\
Next, for each $j=1,2,\ldots$, we have	
{\footnotesize\begin{align*}
&\lambda_j\big((A'!_{\mu} B')^{-\frac{1}{2}}(A'\sigma^* B')(A'!_{\mu} B')^{-\frac{1}{2}}\big)\\
&=\lambda_j\big((A'\sigma^* B')^{\frac{1}{2}}(A'!_\mu B')^{-1}(A'\sigma^* B')^{\frac{1}{2}}\big)\tag{by Lemma \ref{p3}}\\
&=\lambda_j\big((A'\sigma^* B')^{\frac{1}{2}}(A'!_\mu B')^{-1}(A'\sigma^* B')^{\frac{1}{2}}-I\big)+1\\
&=\lambda_j\left((A'\sigma^* B')^{\frac{1}{2}}(A'\sharp B')^{-1}(A'\nabla_{1-\mu} B'-A'\sigma^0 B')(A'\sharp B')^{-1}(A'\sigma^* B')^{\frac{1}{2}}\right)+1\tag{by (\ref{20})}\\	
&=\lambda_j\left((A'\nabla_{1-\mu} B'-A'\sigma^0 B')^{\frac{1}{2}}(A'\sharp B')^{-1}(A'\sigma^* B')(A'\sharp B')^{-1}(A'\nabla_{1-\mu} B'-A'\sigma^0 B')^{\frac{1}{2}}\right)+1\tag{by Lemma \ref{p3}}	
\\
&=\lambda_j\left((A'\nabla_{1-\mu} B'-A'\sigma^0 B')^{\frac{1}{2}}(A'\sigma^0 B')^{-1}(A'\nabla_{1-\mu} B'-A'\sigma^0 B')^{\frac{1}{2}}\right)+1\tag{by (\ref{21})}
\\
&\leq2\lambda_j(A'\nabla_{1-\mu} B'-A'\sigma^0 B')+1\tag{since $(A'\sigma^0 B')^{-1}\leq 2 I$}	
\\
&\leq2\lambda_j(A\nabla_{1-\mu} B-A\sigma^0 B)+1\tag{by (\ref{p}) and Lemma \ref{p2}}
\\
&\leq\lambda_j\left((A\nabla_{1-\mu} B-A\sigma^0 B)^{\frac{1}{2}}(A\sigma^0 B)^{-1}(A\nabla_{1-\mu} B-A\sigma^0 B)^{\frac{1}{2}}\right)+1\tag{since  $2I\leq(A\sigma^0 B)^{-1}$}	
\\
&=\lambda_j\left((A\nabla_{1-\mu} B-A\sigma^0 B)^{\frac{1}{2}}(A\sharp B)^{-1}(A\sigma^* B)(A\sharp B)^{-1}(A\nabla_{1-\mu} B-A\sigma^0 B)^{\frac{1}{2}}\right)+1\tag{by (\ref{21})}
\\
&=\lambda_j\left((A\sigma^* B)^{\frac{1}{2}}(A\sharp B)^{-1}(A\nabla_{1-\mu} B-A\sigma^0 B)(A\sharp B)^{-1}(A\sigma^* B)^{\frac{1}{2}}\right)+1\tag{by Lemma \ref{p3}}
\\
&=\lambda_j\big((A\sigma^* B)^{\frac{1}{2}}(A!_\mu B)^{-1}(A\sigma^* B)^{\frac{1}{2}}-I\big)+1\tag{by (\ref{20})}\\
&=\lambda_j\big((A\sigma^* B)^{\frac{1}{2}}(A!_\mu B)^{-1}(A\sigma^* B)^{\frac{1}{2}}\big)\\
&=\lambda_j\big((A!_{\mu} B)^{-\frac{1}{2}}(A\sigma^*B)(A!_{\mu} B)^{-\frac{1}{2}}\big).\tag{by Lemma \ref{p3}}
\end{align*}}
Thus, the system of inequalities (\ref{d3}) holds.
Now, let $f$ be non-linear and equality holds in (\ref{d3}). Then by the above computation, we get
\begin{equation*}
\lambda_j(A'\nabla_{1-\mu} B'-A'\sigma^0 B')=\lambda_j(A\nabla_{1-\mu} B-A\sigma^0 B),\qquad j=1,2,\ldots.
\end{equation*}
By the same reasoning as in the proof of Theorem \ref{69}, we conclude that $A=B$.
\end{proof}
\begin{corollary}
Let $0\leq\mu\leq1$. If $A, B\in\mathbb{B}^+(\mathscr{H})$ such that $0<A,B\leq\frac{1}{2}I$ and $A-B$ is compact, then for all $j=1,2,\ldots$,
{\footnotesize\begin{align*}
	&\lambda_j\left((A'\sharp_\mu B')^{-\frac{1}{2}}(A'\nabla_\mu B')(A'\sharp_\mu B')^{-\frac{1}{2}}\right)\leq\lambda_j\left((A\sharp_\mu B)^{-\frac{1}{2}}(A\nabla_\mu B)(A\sharp_\mu B)^{-\frac{1}{2}}\right),\\
	&\lambda_j\left((A'!_\mu B')^{-\frac{1}{2}}(A'\sharp_\mu B')(A'!_\mu B')^{-\frac{1}{2}}\right)\leq\lambda_j\left((A!_\mu B)^{-\frac{1}{2}}(A\sharp_\mu B)(A!_\mu B)^{-\frac{1}{2}}\right),\\
	&\lambda_j\left((A'!_\mu B')^{-\frac{1}{2}}(A'\nabla_\mu B')(A'!_\mu B')^{-\frac{1}{2}}\right)\leq\lambda_j\left((A!_\mu B)^{-\frac{1}{2}}(A\nabla_\mu B)(A!_\mu B)^{-\frac{1}{2}}\right).
	\end{align*}}
Moreover, when $0<\mu<1$, equality holds in each of the above systems of inequalities if and only if $A = B$.
\end{corollary}

\textbf{Acknowledgment.}
The authors would like to be grateful to Professor Maryam Khosravi for her valuable comments regarding to eigenvalue inequalities. Also, we thank Professors Hiroyuki Osaka and Shuhei Wada for introducing the example contained in Remark \ref{4}. 


 \end{document}